\documentclass[11pt]{article}
\usepackage{latexsym, amsmath, amssymb}

\usepackage{amsmath}
\usepackage{amsfonts}
\usepackage{amssymb}
\usepackage{amscd}
\usepackage{amsthm}
\usepackage{fancyhdr}
\theoremstyle{plain}
\newtheorem{theorem}{Theorem}[section]
\newtheorem{proposition}[theorem]{Proposition}
\newtheorem{lemma}[theorem]{Lemma}

\theoremstyle{definition}
\newtheorem{remark}[theorem]{Remark}
\newtheorem{definition}[theorem]{Definition}

\date{}

\def\Q{\mathbb Q}

\newcommand\sO{{\mathcal O}}

  \def \tab#1{\kern #1 truein}

  \def\Q{\hbox{${\cal Q}$}}

\begin{document}
 \title{Qregularity and tensor products of vector bundles on smooth quadric hypersurfaces}
\author{Edoardo Ballico and Francesco Malaspina
\vspace{6pt}\\
{\small   Universit\`a di Trento}\\
{\small\it 38050 Povo (TN), Italy}\\
{\small\it e-mail: ballico@science.unitn.it}\\
\vspace{6pt}\\
{\small  Politecnico di Torino}\\
{\small\it  Corso Duca degli Abruzzi 24, 10129 Torino, Italy}\\
{\small\it e-mail: francesco.malaspina@polito.it}}
    \maketitle \def\thefootnote{}
\footnote{
keywords: Spinor bundles; coherent sheaves on quadric hypersurfaces; Castelnuovo-Mumford regularity.}
 \begin{abstract}
 Let $\Q_n \subset \mathbb P^{n+1}$ be a smooth quadric hypersurface. Here we prove that the
 tensor product of an $m$-Qregular sheaf on $\Q_n$ and an $l$-Qregular vector bundle on $\Q_n$ is $(m+l)$-Qregular.
\end{abstract}

\maketitle

 \section{Introduction}
 Let $\Q_n \subset \mathbb P^{n+1}$ be
   a smooth quadric hypersurface. We use the unified notation $\Sigma_*$ meaning that for even $n$ both the spinor bundles $\Sigma_1$ and $\Sigma_2$ are considered, while $\Sigma_\ast = \Sigma$
if $n$ is odd.
 We recall the definition of Qregularity for a coherent sheaf on $\Q_n$ given in \cite{bm}: 
 \begin{definition} A  coherent sheaf $F$ on $\Q_n$ ($n\geq 2$) is said to be $m$-Qregular if one of the following equivalent conditions are satisfied:\begin{enumerate}
 \item $H^i(F(m-i))=0$ for $i=1,\dots ,n-1$, and $H^n(F(m)\otimes \Sigma_*(-n))=0$.\\
 
 \item $H^i(F(m-i))=0$ for $i=1,\dots, n-1$, $H^{n-1}(F(m)\otimes \Sigma_*(-n+1))=0$, and\\
    $H^n(F(m-n+1))=0$.
    \end{enumerate}
  
     In \cite{bm} we defined the Qregularity of $F$, $Qreg (F)$, as the least integer $m$ such that $F$ is $m$-Qregular. We set $Qreg (F)=-\infty$ if there is no such an integer.
  \end{definition}
  
  Here we prove the following property of Qregularity.
 
   \begin{theorem}\label{i1}
   Let $F$ and $G$ be $m$-Qregular and $l$-Qregular coherent sheaves such that $Tor_i(F,G)=0$ for $i>0$. Then $F\otimes G$ is $(m+l)$-Qregular. In particular this holds if one of them is locally free.
    \end{theorem}

The corresponding result is true taking as regularity either the Castelnuovo-Mumford regularity or (for sheaves on a Grassmannian) the Grassmann regularity
defined by J. V. Chipalkatti (\cite{c}, Theorem 1.9). The corresponding result is not true (not even if $G$ is a line bundle) on many varieties with respect to geometric collections
or $n$-block collections (very general and very important definitions of regularity discovered by L. Costa and R.-M. Mir\'{o}-Roig)  (\cite{cm1}, \cite{cm2}, \cite{cm3}). Our definition of Qregularity
on smooth quadric hypersurfaces was taylor-made to get splitting theorems and to be well-behaved with
respect to smooth hyperplane sections. Theorem \ref{i1} gives another good property of it. To get Theorem \ref{i1} we easily adapt Chicalpatti's proof of \cite{c}, Theorem 1.9, except that
we found that in our set-up we need one more vanishing. Our proof of this vanishing shows that on smooth quadric hypersurfaces our definition of Qregularity easily gives
splitting results (see Lemma \ref{a1}).

\section{The proof}\label{S2}

Set $\mathcal {O}:= \mathcal {O}_{\Q_n}$.

 \begin{lemma}Let  $F$  be a $0$-Qregular coherent sheaf on $\Q_n$. Then $F$ admits a finite locally free resolution of the form: $$0\rightarrow K^n\rightarrow\dots \rightarrow K^0\rightarrow F\rightarrow 0,$$ where $K^j$ ($0\leq j<n$) is a finite direct sum of line bundles $\sO(-j)$ and $K^n$ is an $n$-Qregular locally free sheaf.
  \begin{proof} Since $F$ is globally generated (\cite{bm}, proposition 2.5), there is a surjective map $$H^0(F)\otimes \sO\rightarrow F.$$ The kernel $K$ is a coherent sheaf and we have the exact sequence  $$0\to K\to H^0(F)\otimes \sO\rightarrow F\to 0.$$ Since
  the evaluation map  $H^0(F)\otimes \sO\rightarrow F\to 0$ induces a bijection of global sections, $H^1(K)=0$ . From the sequences $$H^{i-1} (F(-i+1))\rightarrow H^i(K(-i+1))\rightarrow H^0(F)\otimes H^i(\sO(-i+1))\rightarrow 0,$$
  we see that $H^i(K(-i+1))=0$ for any $i$ ($1<i<n$).\\
  From the sequences $$H^{n-1} (F)\otimes \Sigma_*(-n+1)\rightarrow H^n(K(1)\otimes \Sigma_*(-n))\rightarrow H^0(F)\otimes H^n(\Sigma_*(-n+1))\rightarrow 0,$$
  we see that $H^n(K(1)\otimes \Sigma_*(-n))=0$. We conclude that $K$ is $1$-Qregular.\\
  We apply the same argument to $K$ and we obtain a surjective map $$H^0(K(1))\otimes \sO(-1)\rightarrow K$$ with a $2$-Qregular kernel. By the syzygies Theorem we obtain the claimed resolution.
  
  \end{proof}
  \end{lemma}
  
  \begin{lemma}\label{a1}
  Let $G$ an $m$-Qregular coherent sheaf on $\Q_n$ such that $h^n(G(-m-n)) \ne 0$. Then $G$ has $\mathcal {O}(-m)$ as a direct factor.
  \end{lemma}
  
  \begin{proof}
  Since $h^n(G(-m-n)) \ne 0$, $h^0(G^\ast (m)) \ne 0$ (\cite{ak}, theorem at page 1). Hence there is a non-zero map $\tau :G(m) \to \mathcal {O}$. Since $G(m)$ is $0$-Qregular, it is spanned
  (\cite{bm}, proposition 2.5), i.e.
  there are an integer $N>0$ and a surjection $u: \mathcal {O}^N \to G(m)$. Every non-zero map
  $\mathcal {O} \to \mathcal {O}$ is an isomorphism. Hence $\tau \circ u$ is surjective and there is $v: \mathcal {O} \to \mathcal {O}^N$
  such that $(\tau \circ u)\circ v$ is the identity map of $\mathcal {O}$. Hence the maps $\tau$ and $v\circ u : \mathcal {O} \to G(m)$ show
  that $G(m) \cong \mathcal {O}\oplus G'$ with $G'\cong \mbox{Ker}(\tau )$.
   \end{proof}
   
   \vspace{0.3cm}
   
\qquad {\emph {Proof of Theorem \ref{i1}.}} We first reduce to the case in which $G$ is indecomposable. Indeed, if $G\cong G_1\oplus G_2$ where $G_1$ is $l$-Qregular and $G_2$ is $l'$-Qregular ($l'\leq l$), then $F\otimes G_1$ is $(l+m)$-Qregular and $F\otimes G_2$ is $(l'+m)$-Qregular ($l'+m\leq l+m$) so $F\otimes G\cong (F\otimes G_1)\oplus (F\otimes G_2)$ is $(l+m)$-Qregular.\\
  We can assume that $G$ is not $\sO(-l)$, because the statement is obviously true in this case. Hence by Lemma \ref{a1} we may assume $H^n(G(l-n))=0$.
   Let us tensorize by $G(l)$ the  resolution of $F(m)$. We obtain the following resolution of $F\otimes G$:
  $$0\rightarrow K^n\otimes G(l)\rightarrow\dots \rightarrow K^0\otimes G(l)\rightarrow F\otimes G(m+l)\rightarrow 0,$$
  where $K^j$ ($0\leq j<n$) is a finite direct sum of line bundles $\sO(-j)$ and $K^n$ is a $n$-Qregular locally free sheaf.\\
  Since $$H^n(G(l-n))=\dots =H^1(G(l-1))=0,$$ we have $H^1(F\otimes G(m+l-1))=0$.\\
  Since $$H^n(G(l-n))=\dots =H^2(G(l-2))=0,$$ we have $H^2(F\otimes G(m+l-2))=0$ and so on.\\
  Moreover, $H^n(G(l)\otimes\Sigma_*(-n))=0$ implies $H^n(F\otimes G(m+l)\otimes\Sigma_*(-n))=0$. Thus $F\otimes G$ is $(m+l)$-Qregular.\qed
 \begin{proposition}\label{p1}
 Let $F$ and $G$ be $m$-Qregular and $l$-Qregular  vector bundles on $\Q_n$. If $F$ is not $(m-1)$-Qregular and $G$ is not $(l-1)$-Qregular then $F\otimes G$ is not
$(m+l-1)$-Qregular. In particular $Qreg(F)=Qreg(G)=0$ implies $Qreg(F\otimes G)=0$.
\begin{proof}By the above argument we can prove the result just for $F$ and $G$ indecomposable. Let us assume that   $G$ is not $(l-1)$-Qregular. We can assume that $G$ is not $\sO(-l)$, because the statement is obviously true in this case. Hence by Lemma \ref{a1} we may assume $H^n(G(l-n))=0$.\\
If $H^i(G(l-i-1))\not=0$ for some $i$ ($0>i>n$), and $$H^{i+1}(G(l-1-i-1))=\dots =H^n(G(l-n))=0,$$ we have an injective map $$H^i(G(l-i-1))\rightarrow H^i(F\otimes G(m+l-i-1))$$ and so $H^i(F\otimes G(m+l-i-1))\not=0$. This means that $F\otimes G$ is not
$(m+l-1)$-Qregular.\\
If $H^i(G(l-i-1))=0$ for any $i$ ($0>i>n$) but $H^{n-1}(G\otimes\Sigma_*(-n))=0$ by \cite{bm} Proof of Theorem $1.2.$, we have that $G\cong \Sigma_*(-l)$. By a symmetric argument we may assume that $F\cong \Sigma_*(-m)$. Now we only need to show that $\Sigma_*(-m)\otimes\Sigma_*(-l)$ is not
$(m+l-1)$-Qregular. Indeed since $h^0(\Sigma_*\otimes\Sigma _*(-1))=0$,
  \cite{bm} Proposition $2.5$ implies that $\Sigma_*\otimes\Sigma _*$ is not $(-1)$-Qregular.

%Indeed $H^n(\Sigma_*(-m)\otimes\Sigma_*(-l-1)\otimes\Sigma_*(l+m-n))\cong H^n(\Sigma_*^\vee\otimes\Sigma_*^\vee\otimes\Sigma_*^\vee(1))\cong H^n(\Sigma_*(-1)\otimes\Sigma_*(-1)\otimes\Sigma_*)
\end{proof}
\end{proposition}

\begin{remark} On $\mathbb P^{n}$ if  $F$ is a regular coherent sheaf accoding Castelnuovo-Mumford, then it admits a finite locally free resolution of the form: $$0\rightarrow K^n\rightarrow\dots \rightarrow K^0\rightarrow F\rightarrow 0,$$ where $K^j$ ($0\leq j<n$) is a finite direct sum of line bundles $\sO(-j)$ and $K^n$ is an $n$-regular locally free sheaf. Now arguing as above we can deduce that Theorem \ref{i1} and Proposition \ref{p1} hold also on $\mathbb P^{n}$ for Castelnuovo-Mumford regularity.
\end{remark}
%We have moreover that if $F$ is not $(m-1)$-Qregular or $G$ is not $(l-1)$-Qregular then $F\otimes G$ is not
%$(m+l-1)$-regular. In particular $Qreg(F)=Qreg(G)=0$ implies $Qreg(F\otimes G)=0$.
%\begin{proof}Let us assume that   $G$ is not $(l-1)$-Qregular. We can assume that $G$ is not $\sO(-l)$, because the statement is obviously true in this case. Hence by Lemma \ref{a1} we may assume $H^n(G(l-n))=0$.\\
%If $H^i(G(l-i))\not=0$ for some $i$ ($0>i>n$), since $$H^{i-1}(G(l-i+1))=\dots =H^1(G(l-1))=0,$$ we have 

\end{document}